\pdfoutput=1 

\documentclass[11pt]{amsart}

\usepackage{amssymb,amsthm,enumitem,colonequals,tikz-cd,microtype}
\usepackage[normalem]{ulem} 

\usepackage[osf]{Baskervaldx}
\usepackage[baskervaldx]{newtxmath}
\usepackage[cal=boondoxo]{mathalfa} 

\usepackage[top=3.75cm, bottom=3cm, left=3.5cm, right=3.5cm]{geometry}

\usepackage{xcolor}
\colorlet{darkblue}{blue!55!black}
\colorlet{darkcyan}{cyan!50!black}
\colorlet{darkgreen}{green!60!black}

\PassOptionsToPackage{hyphens}{url}
\usepackage{hyperref}
\hypersetup{
    colorlinks=true,
    linkcolor=darkblue,
    urlcolor=darkcyan,
    citecolor=darkgreen,
}

\def\eqref#1{\textcolor{darkblue}{(\ref{#1})}}

\usepackage[nameinlink]{cleveref} 
\Crefformat{section}{#2\S#1#3}
\Crefmultiformat{section}{#2\S\S#1#3}{ and~#2#1#3}{, #2#1#3}{, and~#2#1#3}

\usepackage[pagewise]{lineno}
\let\oldequation\equation
\let\oldendequation\endequation
\renewenvironment{equation}{\linenomathNonumbers\oldequation}{\oldendequation\endlinenomath}
\expandafter\let\expandafter\oldequationstar\csname equation*\endcsname
\expandafter\let\expandafter\oldendequationstar\csname endequation*\endcsname
\renewenvironment{equation*}{\linenomathNonumbers\oldequationstar}{\oldendequationstar\endlinenomath}
\let\oldalign\align
\let\oldendalign\endalign
\renewenvironment{align}{\linenomathNonumbers\oldalign}{\oldendalign\endlinenomath}
\expandafter\let\expandafter\oldalignstar\csname align*\endcsname
\expandafter\let\expandafter\oldendalignstar\csname endalign*\endcsname
\renewenvironment{align*}{\linenomathNonumbers\oldalignstar}{\oldendalignstar\endlinenomath}

\theoremstyle{plain}
\newtheorem{theorem}{Theorem}[section]
\newtheorem{lemma}[theorem]{Lemma}
\newtheorem{corollary}[theorem]{Corollary}
\newtheorem{proposition}[theorem]{Proposition}

\theoremstyle{definition}
\newtheorem{definition}[theorem]{Definition}
\newtheorem{example}[theorem]{Example}
\newtheorem{remark}[theorem]{Remark}

\newtheorem{convention}[theorem]{Convention}

\newtheorem*{ack}{Acknowledgments}

\AddToHook{env/conjecture/begin}{\crefalias{theorem}{conjecture}}
\AddToHook{env/lemma/begin}{\crefalias{theorem}{lemma}}
\AddToHook{env/corollary/begin}{\crefalias{theorem}{corollary}}
\AddToHook{env/proposition/begin}{\crefalias{theorem}{proposition}}
\AddToHook{env/definition/begin}{\crefalias{theorem}{definition}}
\AddToHook{env/remark/begin}{\crefalias{theorem}{remark}}
\AddToHook{env/setup/begin}{\crefalias{theorem}{setup}}
\AddToHook{env/example/begin}{\crefalias{theorem}{example}}
\AddToHook{env/convention/begin}{\crefalias{theorem}{convention}}

\numberwithin{equation}{section}
\numberwithin{theorem}{section}

\title[Perfectly generated $t$-structures, algebraic stacks]{Perfectly generated $t$-structures \\ for algebraic stacks}

\author[M.~Hrbek]{Michal Hrbek}
\address{M.~Hrbek,
Institute of Mathematics of the Czech Academy of Sciences,
Žitná 25, 115 67 Prague, Czechia}
\email{hrbek@math.cas.cz}

\author[P.~Lank]{Pat Lank}
\address{P.~Lank,
Dipartimento di Matematica “F. Enriques”, Universit\`{a} degli Studi di Milano, Via Cesare
Saldini 50, 20133 Milano, Italy}
\email{plankmathematics@gmail.com}

\author[S.~Pizzirani]{Simone Pizzirani}
\address{\parbox{\linewidth}{S.~Pizzirani, \\
Institute of Mathematics of the Czech Academy of Sciences,
Žitná 25, 115 67 Prague, Czechia \\
Faculty of Mathematics and Physics, Department
of Algebra, Charles University, Sokolovsk\'a 83, 186 75 Prague, Czechia  }}
\email{pizzirani@math.cas.cz}

\date{\today}



\begin{document}
    
\begin{abstract}
    This work studies $t$-structures for the derived category of quasi-coherent sheaves on suitable algebraic stacks. The main result shows that the standard $t$-structure is compactly generated. As an application, we classify compactly generated tensor $t$-structures via Thomason filtrations on the underlying topological space. 
\end{abstract}

\maketitle

\tableofcontents

\section{Introduction}
\label{sec:intro}

\subsection{What is known}
\label{sec:intro_what_is_known}

Since their introduction in \cite{Beilinson/Berstein/Deligne/Gabber:2018}, $t$-structures have become a fundamental tool for extracting information from derived categories in algebraic geometry and commutative algebra. 
Recent applications include the proof, and generalizations, of the Antieau--Gepner--Heller conjecture \cite{Antieau/Gepner/Heller:2019}. 
On a Noetherian algebraic stack, it has been shown that the category of perfect complexes $\operatorname{Perf}$ admits a bounded $t$-structure if, and only if, the algebraic stack is regular \cite{Smith:2022,Neeman:2024,DeDeyn/Lank/ManaliRahul/Peng:2025}.
 
Roughly speaking, a $t$-structure on a triangulated category $\mathcal{T}$ is a pair of subcategories $(\mathcal{T}^{\leq 0}, \mathcal{T}^{\geq 0})$ satisfying certain conditions. This includes an orthogonality constraint and being able to express objects as extensions of the pair. 
In fact, the $t$-structure is completely determined by its \textit{aisle} $\mathcal{T}^{\leq 0}$ \cite{Keller/Vossieck:1988}. 

For suitable algebraic stacks, it is known that the standard aisle $D^{\leq 0}_{\operatorname{qc}}(\mathcal{X})$ is `equivalent' to a $t$-structure that is singly compactly generated. Here, equivalent means in the sense of \cite[Definition 0.14]{Neeman:2025}. See \cite[Theorem 3.2(iii)]{Neeman:2024} for the case of schemes, and \cite[Proposition 6.6]{DeDeyn/Lank/ManaliRahul/Peng:2025} for suitable algebraic stacks.

More recently, it was shown that the standard aisle need not be itself singly compactly generated for many algebraic stacks \cite{Bhaduri/DeDeyn/Hrbek/Lank/ManaliRahul:2025}. 
On the other hand, for Noetherian schemes, the standard aisle is compactly generated by perfect complexes \cite[Proposition 2.3]{Herbera/Hrbek/LeGros:2025}. Beyond this, however, little appears to be known, even for quasi-compact quasi-separated schemes.

\subsection{What we do}
\label{sec:intro_what_we_do}

\subsubsection{Generation}
\label{sec:intro_what_we_do_generation}

Our first result is the following.

\begin{proposition}
    \label{prop:psuedoapprox}
    Let $\mathcal{X}$ be a concentrated algebraic stack with quasi-finite and separated diagonal (resp.\ a Deligne--Mumford $\mathbb{Q}$-stack). Then $D^{\leq 0}_{\operatorname{qc}}(\mathcal{X})$ is compactly generated.
\end{proposition}

Notably, \Cref{prop:psuedoapprox} appears to be new even in the case of quasi-compact quasi-separated schemes. 

\begin{corollary}
    \label{cor:psuedoapprox}
    Let $X$ be a quasi-compact quasi-separated scheme. Then $D^{\leq 0}_{\operatorname{qc}}(X)$ is compactly generated.
\end{corollary}

To prove \Cref{prop:psuedoapprox}, we introduce a notion called \textit{pseudoapproximation}. 
This allows one to approximate bounded above complexes whose highest nonvanishing cohomology sheaf is finitely presented. 
The condition is weaker than approximation by perfect complexes \cite[Theorem 4.1]{Lipman/Neeman:2007} and weaker than approximation by compact objects \cite{Hall/Lamarche/Lank/Peng:2025}. 
We refer the reader to \Cref{sec:generation} for details. 
Our proof proceeds via \'{e}tale d\'{e}vissage, following \cite{Hall/Rydh:2018}. 
The most technical step appears in \Cref{prop:pseudo_approximation_along_etale_nbhd}, which is concerned with glueing pseudoapproximation along \'{e}tale neighborhoods.

\subsubsection{Classification}
\label{sec:intro_what_we_do_classification}

We apply \Cref{prop:psuedoapprox} to obtain a classification of suitable $t$-structures on $D_{\operatorname{qc}}$. Some historical context is helpful.

The study of $t$-structures on schemes has been very active. 
In full generality, a complete classification of $t$-structures is impossible even for affine schemes \cite{Stanley:2010}. 
Consequently, much of the literature has focused on classes of $t$-structures satisfying additional conditions.

For affine Noetherian schemes, a classification of compactly generated t-structures was obtained in \cite{AlsonoTarrio/JeremiasLopez/Saorin:2010}. Subsequently, \cite{Dubey/Sahoo:2023} extended these results to Noetherian schemes under a tensor compatibility assumption. 
The Noetherian hypothesis was later removed in the affine case by \cite{Hrbek:2020}, and those results play a key role in the present work.

These classifications are formulated in terms of `Thomason filtrations'. Namely, these are decreasing functions from the integers to the collection of Thomason subsets of the underlying topological space. 
See \Cref{sec:prelim_Thomason} for details. 
However, no analogous classification appears to be available for algebraic stacks.

To be precise, an aisle $\mathcal{A}$ on $D_{\operatorname{qc}}(\mathcal{X})$ is called \textit{$\otimes$-closed} if $D^{\leq 0}_{\operatorname{qc}}(\mathcal{X})\otimes^{\mathbf{L}} \mathcal{A} \subseteq \mathcal{A}$. Applying \Cref{prop:psuedoapprox}, this is equivalent to tensoring by perfect complexes which belong to the standard aisle. See \Cref{rmk:proxy_aisle} for details.

Using \Cref{prop:psuedoapprox}, we establish the following classification. 

\begin{theorem}
    \label{thm:stacky_DB}
    Let $\mathcal{X}$ be a concentrated algebraic stack with quasi-finite and separated diagonal (resp.\ a Deligne--Mumford $\mathbb{Q}$-stack). There is a one-to-one correspondence:
    \begin{displaymath}
        \begin{aligned}
            \{ \textrm{compactly generated } \otimes\!\textrm{-aisle on } & D_{\operatorname{qc}}(\mathcal{X})\} 
            \\&\iff \{ \textrm{Thomason filtrations on } \mathcal{X}\}.
        \end{aligned}
    \end{displaymath}
    In particular, for any $\otimes$-aisle on $D_{\operatorname{qc}}(\mathcal{X})$ generated by $\mathcal{P} \subseteq \operatorname{Perf}(\mathcal{X})$, we assign to it the Thomason filtration given by the rule (which is independent of the choice for $\mathcal{P}$)
    \begin{displaymath}
        n\mapsto \bigcup_{\substack{j\geq n\\ P\in \mathcal{P}}} \operatorname{supp}(\mathcal{H}^j (P)).
    \end{displaymath}
\end{theorem}

An important feature of \Cref{thm:stacky_DB} is that using only the affine results of \cite{Hrbek:2020}, we can go directly to algebraic stacks without any dependence on \cite{Dubey/Sahoo:2023}. 
Furthermore, even for quasi-compact quasi-separated schemes, no such classification has been made.

As a first step towards \Cref{thm:stacky_DB}, we show for an affine scheme that every Thomason filtration is supported by perfects. 
Then we establish an auxiliary result (see \Cref{lem:pullback_t_exactness}). 
It allows us to better understand such $t$-structures by passing to an affine scheme that is a smooth cover of the algebraic stack. 
This immediately enables us to bootstrap the classification known in the affine setting.

\begin{ack}
Hrbek and Pizzirani were supported by the project LQ100192601 Lumina quaeruntur, funded by the Czech Academy of Sciences (RVO 67985840). Lank was supported under the ERC Advanced Grant 101095900-TriCatApp, by the National Science Foundation under Grant No.\ DMS-1928930 while in residence at the Simons Laufer Mathematical Sciences Institute (formerly MSRI) in Spring 2024, and thanks the Institute of Mathematics of the Czech Academy of Sciences for its hospitality while conducting parts of this research. The authors greatly appreciate discussions with Alexander Clark, Timothy De Deyn, Giovanna Le Gros, Kabeer Manali Rahul, Chris J.\ Parker, and Sergio Pavon. 
\end{ack}

\section{Preliminaries}
\label{sec:prelim}

\subsection{Algebraic stacks}
\label{sec:prelim_stacks}

We follow the conventions of \cite{StacksProject} for algebraic stacks. 
Symbols such as $X, Y, \ldots$ denote schemes, and $\mathcal{X}, \mathcal{Y}, \ldots$ denote algebraic stacks. 
However, for the derived pullback and pushforward adjunction, we follow \cite[\S1]{Hall/Rydh:2017} and \cite{Olsson:2007a,Laszlo/Olsson:2008a,Laszlo/Olsson:2008b}. 
Let $\mathcal{X}$ be a quasi-compact quasi-separated algebraic stack. 

\subsubsection{Categories}
\label{sec:prelim_stacks_cats}

Define $\operatorname{Mod}(\mathcal{X})$ to be the Grothendieck abelian category of $\mathcal{O}_\mathcal{X}$-modules on the lisse-\'{e}tale topos, and denote by $\operatorname{Qcoh}(\mathcal{X})$ its subcategory of quasi-coherent modules. Set $D(\mathcal{X}) = D(\operatorname{Mod}(\mathcal{X}))$ to be the derived category of $\operatorname{Mod}(\mathcal{X})$, and $D_{\operatorname{qc}}(\mathcal{X})$ its strictly full subcategory of complexes with quasi-coherent cohomology. 

\subsubsection{Concentratedness}
\label{sec:prelim_stacks_concentratedness}

Let $f\colon \mathcal{Y} \to \mathcal{X}$ be a quasi-compact quasi-separated morphism. The formulation of derived pushforward $\mathbf{R}f_\ast$ and pullback functors $\mathbf{L}f^\ast$ on $D_{\operatorname{qc}}$ can be found in \cite[\S 1]{Hall/Rydh:2017}. 
We say $f$ is \textbf{concentrated} if, for any base change along a quasi-compact quasi-separated algebraic stack, $\mathbf{R}f_\ast$ has finite cohomological dimension (see \cite[Lemma 2.5 \& Theorem 2.6]{Hall/Rydh:2017}). 
Furthermore, $\mathcal{X}$ is called \textbf{concentrated} if its structure morphism $\mathcal{X} \to \operatorname{Spec}(\mathbb{Z})$ is concentrated.

\begin{example}
    \label{ex:stacks_concentrated}
    \hfill
    \begin{enumerate}
        \item Let $G$ be a group scheme of finite type over a field $k$. 
        A full classification of when the classifying stack $B_k G$ is concentrated can be found in \cite[Theorem B]{Hall/Rydh:2015}.
        \item Any Noetherian $\mathbb{Q}$-algebraic stack with affine stabilizer groups is concentrated \cite[Theorem 1.4.2]{Drifield/Gaitsgory:2013}. 
        This was generalized to positive characteristic in \cite[Theorem C]{Hall/Rydh:2015}.
        \item More generally, any quasi-compact quasi-separated tame Deligne--Mumford stack is concentrated if it has affine or quasi-compact separated diagonal. 
        Also, any algebraic $\mathbb{Q}$-stack of $s$-global type\footnote{From \cite[\S 2]{Rydh:2015}, $X$ is `$s$-global type' if it is \'{e}tale-locally the quotient of a quasi-affine scheme by
        $\operatorname{GL}_N$ for some $N$.} is concentrated. 
        These follow from the main results of \cite{Hall/Rydh:2017} if coupled with \cite[Theorem C]{Hall/Rydh:2015}. See \cite[\S.\ Perfect Stacks]{Hall/Rydh:2017} for details.
    \end{enumerate}
\end{example}

\subsubsection{Perfect complexes}
\label{sec:prelim_stacks_perfect}

Denote by $\operatorname{Perf}(\mathcal{X}) \subseteq D_{\operatorname{qc}}(\mathcal{X})$ the triangulated subcategory of $D_{\operatorname{qc}}(\mathcal{X})$ consisting of perfect complexes. 
These are defined as objects (on ringed sites) which locally are bounded complexes whose components are direct summands of finite free modules. 
By \cite[Proposition 3.1]{Hall/Neeman/Rydh:2019}, the perfect complexes need not coincide with compact objects of $D_{\operatorname{qc}}$ (loc.\ cit.\ yields a regular Noetherian algebraic stack, and so it is nonzero perfect complexes). 
However, compacts do coincide with perfects if, and only if, $\mathcal{X}$ is concentrated. 
See \cite[Lemma 4.4(3) \& Remark 4.6]{Hall/Rydh:2017}.

\subsubsection{Support}
\label{sec:prelim_stacks_support}

For any $M \in \operatorname{Qcoh}(\mathcal{X})$, define $\operatorname{supp}(M) := p\big(\operatorname{supp}(p^\ast M)\big) \subseteq |\mathcal{X}|$ where $p \colon U \to \mathcal{X}$ is any smooth surjective morphism from a scheme; this is independent of $p$. 
For $E \in D_{\operatorname{qc}}(\mathcal{X})$, set
\begin{displaymath}
    \operatorname{supp}(E) := \bigcup_{j \in \mathbb{Z}} \operatorname{supp}\big(\mathcal{H}^j(E)\big) \subseteq |\mathcal{X}|.
\end{displaymath}
Given a closed subset $Z \subseteq |\mathcal{X}|$, we say $E$ is \textbf{supported on $Z$} if $\operatorname{supp}(E) \subseteq Z$.

\subsubsection{Compact generation and approximation}
\label{sec:prelim_stacks_compact_gen_approx}

We say $\mathcal{X}$ satisfies the
\textbf{$\beta$-Thomason condition}, for some cardinal $\beta$, if $D_{\operatorname{qc}}(\mathcal{X})$ is compactly generated by a collection of cardinality at most $\beta$ and for each closed subset $Z$ of $|\mathcal{X}|$ with quasi-compact complement, there is a $P\in \operatorname{Perf}(\mathcal{X})$ such that $\operatorname{supp}(P)=Z$. 

\begin{example}
    \label{ex:stacks_Thomason}
    \hfill
    \begin{enumerate}
        \item Any quasi-compact quasi-separated algebraic stack with quasi-finite separated diagonal is $1$-Thomason \cite[Theorem A]{Hall/Rydh:2017}.
        \item Any quasi-compact quasi-separated algebraic $\mathbb{Q}$-stack that is of $s$-global type must be $|\mathbb{Z}|$-Thomason \cite[Theorem B]{Hall/Rydh:2017} (i.e.\ countably generated).
    \end{enumerate}
\end{example}



\subsubsection{Mayer--Vietoris squares for algebraic stacks}
\label{sec:prelim_MV}

Consider a fibered square of quasi-compact quasi-separated algebraic stacks
\begin{equation}
    \label{eq:MV}    
    \begin{tikzcd}
        {\mathcal{U}^\prime} & {\mathcal{X}^\prime} \\
        {\mathcal{U}} & {\mathcal{X}}
        \arrow["{j^\prime}", from=1-1, to=1-2]
        \arrow["{f^\prime}"', from=1-1, to=2-1]
        \arrow["f", from=1-2, to=2-2]
        \arrow["j"', from=2-1, to=2-2]
    \end{tikzcd}
\end{equation}
where $j$ is an open immersion and $f$ is a concentrated flat morphism. 
We say that \eqref{eq:MV} is a \textbf{flat Mayer--Vietoris square} if for every morphism $\mathcal{W}\to \mathcal{X}$ with topological image disjoint from $|\mathcal{U}|$, the induced morphism $\mathcal{X}^\prime \times_{\mathcal{X}} \mathcal{W} \to \mathcal{W}$ is an isomorphism. See \cite[Definition 1.2]{Hall/Rydh:2023}. 
Note loc.\ cit.\ does not require concentrated. 
Instead, we impose this condition because it will be needed in later arguments. 
Moreover, we say \eqref{eq:MV} is an \textbf{\'{e}tale neighbourhood} if $f$ is \'{e}tale and the base change of $f$ along the closed immersion $\mathcal{Z}_{red} \to \mathcal{X}$ is an isomorphism where $|\mathcal{Z}_{red}| = |\mathcal{X}|\setminus |\mathcal{U}|$ (cf.\ \cite{Rydh:2011}). 
By \cite[Lemma 2.1]{Rydh:2011}, being an \'{e}tale neighborhood is equivalent to the condition $f$ is \'{e}tale and that \eqref{eq:MV} is a flat Mayer--Vietoris square. 
In this case, \cite[Proposition 4.2]{Hall/Rydh:2023} says that $f$ induces an adjoint $t$-exact equivalence,
\begin{displaymath}
    D_{\operatorname{qc},\mathcal{Z}}(\mathcal{X}) 
    \cong D_{\operatorname{qc},\mathcal{Z}^\prime}(\mathcal{X}^\prime). 
\end{displaymath}
Here, $t$-exactness refers to the respective standard $t$-structures. 
Specifically, the adjoint equivalence is given by the restrictions of $\mathbf{L}f^\ast$ and $\mathbf{R}f_\ast$. 
Moreover, if $P$ is in $D_{\operatorname{qc},\mathcal{Z}^\prime}(\mathcal{X}^\prime) \cap D_{\operatorname{qc}}(\mathcal{X}^\prime)^c$, then $\mathbf{R}f_\ast P$ is in $D_{\operatorname{qc},\mathcal{Z}}(\mathcal{X}) \cap D_{\operatorname{qc}}(\mathcal{X})^c$ \cite[Lemma 5.9]{Hall/Rydh:2017}. 
In particular, if $E \in D_{\operatorname{qc}}(\mathcal{X})$ there is an induced long exact sequence:
\begin{equation}
    \label{eq:fmv-LES}
    \begin{aligned}
        \cdots 
        &\to \operatorname{Ext}^d_{\mathcal{O}_{\mathcal{X}}}(E,M)
        \to \substack{\operatorname{Ext}^d_{\mathcal{O}_{\mathcal{X}^\prime}}(\mathbf{L}f^\ast E,\mathbf{L}f^\ast M) \\ \bigoplus \\  \operatorname{Ext}^d_{\mathcal{O}_{\mathcal{U}}}(\mathbf{L}j^\ast E,\mathbf{L}j^\ast M)}
        \\&\to \operatorname{Ext}^d_{\mathcal{O}_{\mathcal{U}^\prime}}(\mathbf{L}k^\ast E,\mathbf{L} k^\ast M) 
        \to \operatorname{Ext}^{d+1}_{\mathcal{O}_{\mathcal{X}}}(E,M) 
        \to \cdots
    \end{aligned}
\end{equation}
where $k= j \circ f^\prime$.

\subsection{\texorpdfstring{$t$}{t}-structures}
\label{sec:prelim_t-structures}

We recall $t$-structures on triangulated categories and their tensor variants. 
Let $\mathcal{T}$ be a triangulated category admitting small coproducts. 
Its subcategory of compact objects is denoted by $\mathcal{T}^c$, consisting of those $c \in \mathcal{T}$ for which $\operatorname{Hom}(c,-)$ preserves coproducts.

\subsubsection{\texorpdfstring{$t$}{t}-structures}
\label{sec:prelim_t-structures_t_structures}

A pair of strictly full subcategories $(\mathcal{T}^{\leq 0}, \mathcal{T}^{\geq 0})$ of $\mathcal{T}$ is a \textbf{$t$-structure} if:
\begin{itemize}
    \item $\operatorname{Hom}(A,B) = 0$ for all $A \in \mathcal{T}^{\leq 0}$ and $B \in \mathcal{T}^{\geq 0}[-1]$,
    \item $\mathcal{T}^{\leq 0}[1] \subseteq \mathcal{T}^{\leq 0}$ and $\mathcal{T}^{\geq 0}[-1] \subseteq \mathcal{T}^{\geq 0}$,
    \item for every $E \in \mathcal{T}$, there is a distinguished triangle
    \begin{displaymath}
        \tau^{\leq 0} E \to E \to \tau^{\geq 1} E \to (\tau^{\leq 0} E)[1]
    \end{displaymath}
    with $\tau^{\leq 0} E \in \mathcal{T}^{\leq 0}$ and $\tau^{\geq 1} E \in \mathcal{T}^{\geq 0}[-1]$.
\end{itemize}
This notion was introduced in \cite{Beilinson/Berstein/Deligne/Gabber:2018}. 
The above triangle is called the \textbf{truncation triangle} of $E$ with respect to the $t$-structure. 
For any $n \in \mathbb{Z}$, define $\mathcal{T}^{\leq n} := \mathcal{T}^{\leq 0}[-n]$ and $\mathcal{T}^{\geq n} := \mathcal{T}^{\geq 0}[-n]$. 
Then $(\mathcal{T}^{\leq n}, \mathcal{T}^{\geq n})$ is also a $t$-structure on $\mathcal{T}$.

\begin{example}
    Let $\mathcal{X}$ be a quasi-compact quasi-separated algebraic stack. Consider the following categories
    \begin{itemize}
        \item $D^{\leq 0}(\mathcal{X})$ consists of objects $E\in D(\mathcal{X})$ such that $\mathcal{H}^j (E)=0$ for $j >0$
        \item $D^{\geq 0}(\mathcal{X})$ consists of objects $E\in D(\mathcal{X})$ such that $\mathcal{H}^j (E)=0$ for $j <0$. 
    \end{itemize}
    The pair $(D^{\leq 0}(\mathcal{X}), D^{\geq 0}(\mathcal{X}))$ is called the \textbf{standard $t$-structure} on $D(\mathcal{X})$. 
    From \cite[\href{https://stacks.math.columbia.edu/tag/06WU}{Tag 06WU}]{StacksProject}, the natural functor $\operatorname{Qcoh}(\mathcal{X}) \to \operatorname{Mod}(\mathcal{X})$ is fully faithful (but not exact generally). 
    It follows that the standard $t$-structure induces a $t$-structure on $D_{\operatorname{qc}}(\mathcal{X})$. 
    Specifically, the pair
    $(D^{\leq 0}(\mathcal{X})\cap D_{\operatorname{qc}}(\mathcal{X}), D^{\geq 0}(\mathcal{X})\cap D_{\operatorname{qc}}(\mathcal{X}))$ is called the \textbf{standard $t$-structure} on $D_{\operatorname{qc}}(\mathcal{X})$. 
    To simplify, the induced $t$-structure is denoted $(D^{\leq 0}_{\operatorname{qc}}(\mathcal{X}), D^{\geq 0}_{\operatorname{qc}}(\mathcal{X}))$.
\end{example}

\subsubsection{(Pre)aisles}
\label{sec:prelim_t-structures_aisles}

A strictly full subcategory $\mathcal{A} \subseteq \mathcal{T}$ is an \textbf{aisle} if it is closed under positive shifts and extensions, and the inclusion $\mathcal{A} \hookrightarrow \mathcal{T}$ admits a right adjoint. 
Given a $t$-structure $(\mathcal{T}^{\leq 0}, \mathcal{T}^{\geq 0})$, the subcategory $\mathcal{T}^{\leq 0}$ is an aisle. 
By \cite{Keller/Vossieck:1988}, any aisle $\mathcal{A}$ determines a $t$-structure $(\mathcal{A}, \mathcal{A}^\perp[1])$, where
\begin{displaymath}
    \mathcal{A}^\perp := \{ T \in \mathcal{T} \mid \operatorname{Hom}(A,T) = 0 \text{ for all } A \in \mathcal{A} \}.
\end{displaymath}
The subcategories $\mathcal{T}^{\leq 0}$ and $\mathcal{T}^{\geq 0}$ are called the \textbf{aisle} and \textbf{coaisle} of the $t$-structure, respectively. 
More generally, a \textbf{preaisle} is a strictly full subcategory closed under positive shifts and extensions. 
A preaisle $\mathcal{A} \subseteq \mathcal{T}$ is called \textbf{cocomplete} if it is closed under all coproducts in $\mathcal{T}$. 
Given $\mathcal{S}\subseteq \mathcal{T}$, $\overline{\langle \mathcal{S} \rangle}^{(-\infty, 0]}$ is defined to be the smallest cocomplete preaisle containing $\mathcal{S}$. 
By \cite[Theorem 2.3]{Neeman:2021}, if $\mathcal{T}$ is well generated (e.g.\ compactly generated), $\overline{\langle \mathcal{S} \rangle}^{(-\infty, 0]}$ is an aisle whenever $\mathcal{S}$ is essentially small. 
An aisle $\mathcal{U}$ on $\mathcal{T}$ is \textbf{compactly generated} when there exists a collection of compact objects $\mathcal{P} \subseteq \mathcal{T}^c$ satisfying $\overline{\langle \mathcal{P} \rangle}^{(-\infty, 0]} = \mathcal{U}$. 
Hence, we say a $t$-structure is \textbf{compactly generated} if its aisle is as such.  

\subsubsection{\texorpdfstring{$t$}{t}-exactness}
\label{sec:prelim_t-structures_t_exactness}

Let $F \colon \mathcal{T} \to \mathcal{T}^\prime$ be an exact functor between triangulated categories equipped with $t$-structures $(\mathcal{T}^{\leq 0}, \mathcal{T}^{\geq 0})$ and $((\mathcal{T}^\prime)^{\leq 0}, (\mathcal{T}^\prime)^{\geq 0})$. 
We say that $F$ is \textbf{right $t$-exact} if $
F(\mathcal{T}^{\leq 0}) \subseteq (\mathcal{T}^\prime)^{\leq 0}$,
and \textbf{left $t$-exact} if 
$F(\mathcal{T}^{\geq 0}) \subseteq (\mathcal{T}^\prime)^{\geq 0}$. 
If both conditions hold, then $F$ is \textbf{$t$-exact}.
For example, if $f \colon U \to \mathcal{X}$ is a flat morphism from an affine Noetherian scheme to a Noetherian algebraic stack, the derived pullback functor $\mathbf{L}f^\ast \colon D_{\operatorname{qc}}(\mathcal{X}) \to D_{\operatorname{qc}}(U)$ is $t$-exact with respect to the standard $t$-structures. 
At times, we might write a (resp.\ right, left) $t$-exact functor as $F\colon (\mathcal{T},\mathcal{T}^{\leq 0}) \to (\mathcal{T}^\prime ,(\mathcal{T}^\prime)^{\leq 0})$ where only the aisle is made clear.

\subsubsection{Tensor variants}
\label{sec:prelim_t-structures_tensor}

We discuss a notion of aisles via tensor actions of triangulated categories. 
See \cite[\S 3]{Stevenson:2013} for the definition of tensor action. 
The following appears in \cite[\S 4]{Hrbek/Lank/LeGros/Pavon:2026}. 
Let $(\mathcal{T},\otimes,\mathbf{1})$ be a rigidly compactly generated tensor triangulated category and $\mathcal{K}$ be a compactly generated triangulated category. 
Suppose $\odot\colon \mathcal{T}\times\mathcal{K}\to \mathcal{K}$ is an action of $\mathcal{T}$ on $\mathcal{K}$. 
Note that $\odot$ is exact and coproduct preserving in both variables. 
By \cite[Lemma 4.6]{Stevenson:2013}, the action restricts to an action $\odot\colon \mathcal{T}^c\times\mathcal{K}^c\to \mathcal{K}^c$.  
Fix a preaisle $\mathcal{P}^{\leq 0}\subseteq\mathcal{T}^c$ satisfying
$\mathbf{1}\in\mathcal{P}^{\leq 0}$ and $\mathcal{P}^{\leq 0}\otimes\mathcal{P}^{\leq 0}\subseteq\mathcal{P}^{\leq 0}$. 
Set $\mathcal{T}^{\leq 0}\colonequals \overline{\langle \mathcal{P}^{\leq 0} \rangle}^{(-\infty, 0]}$ to be the aisle generated by $\mathcal{P}^{\leq 0}$.  
We say that a preaisle $\mathcal{U}\subseteq\mathcal{K}$ is a \textbf{$\odot$-preaisle} if $\mathcal{P}^{\leq 0}\odot\mathcal{U}\subseteq\mathcal{U}$. 
By \cite[Lemma 2.6]{Hrbek/Lank/LeGros/Pavon:2026}, for a cocomplete preaisle  $\mathcal{A}\subseteq\mathcal{K}$, $\mathcal{A}$ is a $\odot$-preaisle if, and only if, $\mathcal{T}^{\leq 0}\odot \mathcal{A}\subseteq\mathcal{A}$. 
Let $\mathcal{S}\subseteq\mathcal{K}$ be a set of objects. 
Denote by $\overline{\langle \mathcal{S} \rangle}^{(-\infty, 0]}_\odot$ the smallest cocomplete $\odot$-preaisle of $\mathcal{K}$ containing $\mathcal{S}$ (which is obtained by intersecting all of them). Given a skeletally small subcategory $\mathcal{S}\subseteq\mathcal{K}$, $\overline{\langle \mathcal{S} \rangle}^{(-\infty, 0]}_\odot = \overline{\langle \mathcal{P}^{\leq 0}\odot\mathcal{S} \rangle}^{(-\infty, 0]}$. 
In fact, $\overline{\langle \mathcal{S} \rangle}^{(-\infty, 0]}_\odot$ is an aisle. 
See \cite[Lemma 2.8]{Hrbek/Lank/LeGros/Pavon:2026}.


\begin{lemma}
    \label{lem:dubey_sahoo_corrected}
    Let $(\mathcal{T},\otimes,1)$ be a rigidly compactly generated tensor triangulated category.  Suppose $\mathcal{T}^{\leq 0}$ is a compactly generated preaisle. Assume that $1\in \mathcal{K}:= \mathcal{T}^c \cap \mathcal{T}^{\leq 0}$. For any $\mathcal{S}\subseteq \mathcal{T}^c$ and $A\in \overline{\langle \mathcal{S} \rangle}^{(-\infty,0]}_{\otimes}$, there is a distinguished triangle 
    \begin{displaymath}
        \bigsqcup_{i\geq 0} A_i \to A \to (\bigsqcup_{i\geq 0} A_i)[1] \to (\bigsqcup_{i\geq 0} A_i)[1]
    \end{displaymath}
    where $A_i$ are $i$-fold extensions of small coproducts of nonnegative shifts of objects in $\mathcal{K}\otimes \mathcal{S}$.
\end{lemma}

\begin{proof}
    This is exactly \cite[Proposition 3.11]{Dubey/Sahoo:2023} with the correction that $1\in \mathcal{K}:= \mathcal{T}^c \cap \mathcal{T}^{\leq 0}$ is added as a hypothesis.
\end{proof}

\subsection{Thomason filtrations}
\label{sec:prelim_Thomason}

We recall Thomason filtrations on a topological space $X$. 
Let $\operatorname{Spcl}(X)$ be the collection of specialization closed subsets of $X$. 
Recall that $Z\subseteq X$ is called a \textbf{Thomason subset} of $X$ if $Z=\cup_{\alpha \in \Lambda} Z_{\alpha}$ where each $Z_{\alpha}$ is closed in $X$ and $X\setminus Z_{\alpha}$ is quasi-compact. 
Denote by $\operatorname{Thom}(X)$ the collection of Thomason subsets of $X$. 
If $X$ is a Noetherian topological space, then specialization closed subsets coincide with Thomason subsets. 
A function $\phi \colon \mathbb{Z} \to \operatorname{Thom}(X)$ is called a \textbf{Thomason filtration} on $X$ if $\phi(n+1)\subseteq \phi (n)$ for each $n\in \mathbb{Z}$. 
If $X=|\mathcal{X}|$ where $\mathcal{X}$ is an algebraic stack, then we abuse language and call these `Thomason filtrations on $\mathcal{X}$'.

\section{Generation}
\label{sec:generation}

\begin{definition}
    \label{def:approximation}
    Let $\mathcal{X}$ be an algebraic stack. 
    Consider a pair $(Z, E)$  consisting of a closed subset $Z\subseteq|\mathcal{X}|$ and a complex $E\in D^{\leq 0}_{\operatorname{qc},Z}(\mathcal{X})$ such that $\mathcal{H}^0 (E)$ is finitely presented. 
    We say that \textbf{pseudoapproximation by compact complexes holds for $(Z, E)$} if there exist $P\in D^{\leq 0}_{\operatorname{qc},Z}(\mathcal{X}) \cap D_{\operatorname{qc}}(\mathcal{X})^c$ and 
    $P\to E$ such that $\mathcal{H}^0 (P) \to \mathcal{H}^0 (E)$ is surjective. 
    Moreover, we say that \textbf{pseudoapproximation by compact complexes holds
    for $\mathcal{X}$} if for every quasi-compact open immersion $\mathcal{U}\to \mathcal{X}$, pseudoapproximation by compact complexes
    holds for any pair of the form $(Z, E)$ where $Z:=|\mathcal{X}| \setminus |\mathcal{U}|$ and $E\in D^{\leq 0}_{\operatorname{qc},Z}(\mathcal{X})$ such that $\mathcal{H}^0 (E)$ is finitely presented.
\end{definition}

\begin{example}
    \label{ex:quasi-affine_pseudoapprox}
    An affine scheme satisfies pseudoapproximation by compacts. 
    In fact, any quasi-affine scheme does as well. 
    Indeed, it satisfies the resolution property. 
    More generally, let $\mathcal{X}$ be quasi-compact quasi-separated algebraic stack with finite stabilizer groups. 
    Then $\mathcal{X}$ having the resolution property is equivalent to $\mathcal{X}$ being isomorphic to a quotient stack of the form $[U/\operatorname{GL}_{n,\mathbb{Z}}]$ where $U$ is some quasi-affine scheme. 
    This is \cite[Proposition 7.3]{Hall/Rydh:2017}. See Example 7.5 of loc.\ cit.\ for concrete cases where it applies.
\end{example}

\begin{lemma}
    \label{lem:approx_by_comorphismcts_implies_Thomason}
    Let $\mathcal{X}$ be a concentrated algebraic stack quasi-finite and separated diagonal (resp.\ a Deligne--Mumford $\mathbb{Q}$-stack). 
    If $\mathcal{X}$ satisfies pseudoapproximation by compacts, then $D^{\leq 0}_{\operatorname{qc},Z}(\mathcal{X})$ is compactly generated for any closed $Z\subseteq |\mathcal{X}|$. 
\end{lemma}

\begin{proof}
    Fix a closed subset $Z$ of $|\mathcal{X}|$. Set $\mathcal{P}=D_{\operatorname{qc},Z}(\mathcal{X})^c \cap D^{\leq 0}_{\operatorname{qc},Z}(\mathcal{X})$. 
    By construction, $\overline{\langle \mathcal{P} \rangle}^{(-\infty,0]} \subseteq D^{\leq 0}_{\operatorname{qc},Z}(\mathcal{X})$. 
    We prove the reverse inclusion. 
    This is equivalent to showing that $(\overline{\langle \mathcal{P} \rangle}^{(-\infty,0]})^\perp \subseteq D^{\geq 0}_{\operatorname{qc},Z}(\mathcal{X})$, which will be done by contradiction.

    Let $G$ be a compact generator for $D_{\operatorname{qc},Z}(\mathcal{X})$. 
    If necessary, we can impose that $G\in D^{\leq 0}_{\operatorname{qc},Z}(\mathcal{X})$. 
    Before proving the claim, we make the observation that $(\overline{\langle \mathcal{P} \rangle}^{(-\infty,0]})^\perp \subseteq D^+_{\operatorname{qc},Z}(\mathcal{X})$. 
    Indeed, if $G\in D^{\leq 0}_{\operatorname{qc},Z}(\mathcal{X})$, then $\overline{\langle G \rangle}^{(-\infty,0]} \subseteq \overline{\langle \mathcal{P} \rangle}^{(-\infty,0]}$. 
    By \cite[Proposition 6.6]{DeDeyn/Lank/ManaliRahul/Peng:2025}, there exists $n\geq 0$ such that 
    \begin{displaymath}
        D^{\leq -n}_{\operatorname{qc},Z}(\mathcal{X}) 
        \subseteq \overline{\langle G \rangle}^{(-\infty,0]} 
        \subseteq D^{\leq n}_{\operatorname{qc},Z}(\mathcal{X}).
    \end{displaymath}
    Hence, $D^{\leq -n}_{\operatorname{qc},Z}(\mathcal{X}) \subseteq \overline{\langle \mathcal{P} \rangle}^{(-\infty,0]}$, and so,
    \begin{displaymath}
        (\overline{\langle \mathcal{P} \rangle}^{(-\infty,0]})^\perp 
        \subseteq D^{\geq -n}_{\operatorname{qc},Z}(\mathcal{X}) 
        \subseteq D^+_{\operatorname{qc},Z}(\mathcal{X}).
    \end{displaymath}

    Now, we prove that $D^{\leq 0}_{\operatorname{qc},Z}(\mathcal{X})$ is compactly generated. 
    By assumption, $(\overline{\langle \mathcal{P} \rangle}^{(-\infty,0]})^\perp \not\subseteq D^{\geq 0}_{\operatorname{qc},Z}(\mathcal{X})$. 
    This gives us some nonzero $E\in (\overline{\langle \mathcal{P} \rangle}^{(-\infty,0]})^\perp$ such that $E \not\in D^{\geq 0}_{\operatorname{qc},Z}(\mathcal{X})$. 
    Since $(\overline{\langle \mathcal{P} \rangle}^{(-\infty,0]})^\perp \subseteq D^+_{\operatorname{qc},Z}(\mathcal{X})$, there exists a minimal $s\leq 0$ such that $\mathcal{H}^s (E)\not=0$. 
    By \cite[Corollary B]{Rydh:2023}, there exists finitely presented $\mathcal{O}_{\mathcal{X}}$-modules $E_i$ such that $\mathcal{H}^s (E)$ is the directed colimit of the $E_i$. 
    Consequently, we can find some $E_t$ such that $\operatorname{Hom}(E_t,\mathcal{H}^s (E))\not=0$. 
    As $\mathcal{X}$ satisfies pseudoapproximation by compacts, there exists $P\in D^{\leq 0}_{\operatorname{qc},Z}(\mathcal{X}) \cap D_{\operatorname{qc}, Z}(\mathcal{X})^c$ and $P\to E_t$ such that $\mathcal{H}^0 (P) \to \mathcal{H}^0 (E_t)$ is surjective. 
    Thus, we obtain
    \begin{displaymath}
        \begin{aligned}
            \operatorname{Hom}(P[-s],E) 
            &\cong \operatorname{Hom}(\mathcal{H}^s (P[-s]), \mathcal{H}^s (E))
            \\&\cong \operatorname{Hom}(\mathcal{H}^0 (P), \mathcal{H}^s (E)) 
            \\&\not\cong 0.
        \end{aligned}
    \end{displaymath}
    Indeed, the morphisms provided above imply the composition $\mathcal{H}^0 (P) \to E_t \to \mathcal{H}^s (E)$ is nonzero. 
    However, $P[-s]\in \overline{\langle \mathcal{P} \rangle}^{(-\infty,0]}$, whereas $E\in (\overline{\langle \mathcal{P} \rangle}^{(-\infty,0]})^\perp$, which is absurd.
\end{proof}

\begin{lemma}
    \label{lem:lisse-etale_pullback_and_qc_pullback}
    Let $f\colon \mathcal{Y}\to \mathcal{X}$ be a concentrated smooth morphism of algebraic stacks. 
    Then the restriction of $\mathbf{L}f^\ast_{\textrm{lis-\'{e}t}}$ to $D_{\operatorname{qc}}(\mathcal{X})$ coincides with $\mathbf{L}f^\ast$. 
    Moreover, if $E\in D_{\operatorname{qc}}(\mathcal{X})$ has finitely presented cohomology in degree $n$, then $\mathbf{L}f^\ast E$ has finitely presented cohomology in degree $n$.
\end{lemma}

\begin{proof}
    The first claim is \cite[Lemma 3.5]{GuisadoVillalgordo/Lank:2026}.
    By \cite[\href{https://stacks.math.columbia.edu/tag/03DO}{Tag 03DO}]{StacksProject}, the second claim follows; e.g.\ use the first claim and that $\mathcal{H}^n (\mathbf{L}f^\ast E) \cong f^\ast \mathcal{H}^n(E)$ for all $n\in \mathbb{Z}$ via flatness.
\end{proof}

\begin{proposition}
    \label{prop:pseudo_approximation_along_etale_nbhd}
    Consider a flat Mayer--Vietoris square as in \eqref{eq:MV}. 
    If pseudoapproximation by compact complexes holds for $\mathcal{U}$ and $\mathcal{X}^\prime$, then it holds for $\mathcal{X}$. 
\end{proposition}

\begin{proof}
    We adapt the proof of \cite[Proposition 5.2]{Hall/Lamarche/Lank/Peng:2025}.
    If needed, we can replace $\mathcal{X}^\prime$ with $\mathcal{X}^\prime \amalg \mathcal{U}$. 
    This allows us to impose that $f$ is faithfully flat. Let $T\subseteq |\mathcal{X}|$ be a closed subset with quasi-compact complement $V$. Set $T^\prime := T\setminus |\mathcal{U}|$. Fix $E\in D^{\leq 0}_{\operatorname{qc},T}(\mathcal{X})$ such that $\mathcal{H}^0 (E)$ is finitely presented.

    Applying the hypothesis, \cite[Theorem A]{Hall/Rydh:2017} implies $\mathcal{U}$ and $\mathcal{X}^\prime$ satisfy the $1$-Thomason condition. 
    Then, by \cite[Proposition 3.1]{Lank:2026}, $\mathbf{L}j^\ast \colon D_{\operatorname{qc},T}(\mathcal{X}) \to D_{\operatorname{qc},|\mathcal{U}|\cap T}(\mathcal{U})$ is a Verdier localization. 
    Hence, $\mathbf{L}j^{\ast}E$ is supported on $T\cap|\mathcal{U}|$. 
    Furthermore, \Cref{lem:lisse-etale_pullback_and_qc_pullback} says $\mathcal{H}^0 (\mathbf{L}j^{\ast}E)$ is finitely presented.

    Since $\mathcal{U}$ satisfies pseudoapproximation by compacts, there exist $Q\in D_{\operatorname{qc},T \cap |\mathcal{U}|}(\mathcal{U})^c \cap D^{\leq 0}_{\operatorname{qc},T \cap |\mathcal{U}|}(\mathcal{U})$ on $\mathcal{U}$ and $Q\to \mathbf{L}j^{\ast}E$ such that $\mathcal{H}^0(Q)\to \mathcal{H}^0(\mathbf{L}j^{\ast}E)$ is surjective. 
    Once more, by \cite[Proposition 3.1]{Lank:2026}, there exists a Verdier localization sequence 
    \begin{displaymath}
        D_{\operatorname{qc},f^{-1}(T^\prime)}(\mathcal{X}^\prime)
        \to D_{\operatorname{qc},f^{-1}(T)}(\mathcal{X}^\prime)
        \xrightarrow{\mathbf{L}(j^\prime)^\ast} D_{\operatorname{qc},f^{-1}(T) \cap \mathcal{U}^\prime}(\mathcal{U}^\prime).
    \end{displaymath}
    Moreover, from \cite[Remark 3.2]{Lank:2026}, we know that
    \begin{displaymath}
        D_{\operatorname{qc},f^{-1}(T)}(\mathcal{X}^\prime)^c = D_{\operatorname{qc}}(\mathcal{X}^\prime)^c \cap D_{\operatorname{qc},f^{-1}(T)}(\mathcal{X}^\prime).
    \end{displaymath}

    Applying \cite[Theorem 2.1]{Neeman:1996} (which draws from \cite[\S2]{Neeman:1992b} and \cite[\S 5.1]{Thomason/Trobaugh:1990}), we can find a $P^\prime \in D_{\operatorname{qc},f^{-1}(T)}(\mathcal{X}^\prime)^c$ such that $\mathbf{L}(j^{\prime})^\ast P^\prime \cong \mathbf{L}(f^\prime)^\ast(Q\oplus Q[1])$. 
    As $Q\in D^{\leq 0}_{\operatorname{qc},T \cap |\mathcal{U}|}(\mathcal{U})$, we obtain $\tau^{\geq 1} \mathbf{L}(j^{\prime})^\ast P^\prime \cong 0$. 
    Consider the composition
    \begin{displaymath}
        \mathbf{L}(j^{\prime})^\ast P^\prime
        \cong \mathbf{L}(f^\prime)^\ast(Q\oplus Q[1])
        \to \mathbf{L}(f^\prime)^{\ast}Q
        \to \mathbf{L}(f^\prime)^{\ast}\mathbf{L}j^{\ast}E
        \cong\mathbf{L}(j^{\prime})^\ast \mathbf{L}f^{\ast}E.
    \end{displaymath}
    By \cite[Theorem 2.1.5]{Neeman:1996}, this composition in the Verdier quotient corresponds to a roof
    \begin{displaymath}
        \begin{tikzcd}
            & {P^{\prime \prime} } & \\
            {P^\prime} && {\mathbf{L}f^\ast E}
            \arrow["\lambda"', from=1-2, to=2-1]
            \arrow[from=1-2, to=2-3]
        \end{tikzcd}
    \end{displaymath}
    such that $P^{\prime \prime}\in D_{\operatorname{qc},f^{-1}(T)}(\mathcal{X}^\prime)^c$ and $\operatorname{cone}(\lambda) \in D_{\operatorname{qc},f^{-1}(T^\prime)}(\mathcal{X}^\prime)$.

    Appealing to \cite[Example 5.2 \& Lemma 5.9(3,4)]{Hall/Rydh:2017}, there exists $P \in D_{\operatorname{qc},T}(\mathcal{X})^c$ such that $\mathbf{L}f^\ast P \cong P^{\prime \prime}$ and $\mathbf{L}j^\ast P \cong Q\oplus Q[1]$. 
    Using \eqref{eq:fmv-LES}, we may find an exact sequence,
    \begin{displaymath}
        \begin{aligned}
            \operatorname{Hom} (P,E)
            & \to\operatorname{Hom} (\mathbf{L}j^{\ast}P, \mathbf{L}j^{\ast}E)\oplus\operatorname{Hom} (\mathbf{L}f^{\ast}P,\mathbf{L}f^{\ast}E)
            \\&\to\operatorname{Hom} (\mathbf{L}(f^\prime)^{\ast}\mathbf{L}j^{\ast}P,\mathbf{L}(f^\prime)^{\ast}\mathbf{L}j^{\ast}E).
        \end{aligned}
    \end{displaymath}
    Hence, the morphisms $Q \oplus Q[1] \to \mathbf{L} j^{\ast}E$ 
    and $P^{\prime \prime} \to \mathbf{L} f^\ast E$ 
    give rise to a morphism $P \to E$. 
    Consider the distinguished triangle
    \begin{equation}
        \label{eq:pseudo_approximation_along_etale_nbhd_extended}
        P \to E \to E^\prime \to P[1].
    \end{equation}
    By restricting to $\mathcal{U}$, we obtain a diagram
    \begin{displaymath}
        \begin{tikzcd}
            {Q[1]} & 0 & {Q[2]} & {Q[2]} \\
            {Q\oplus Q[1]} & {\mathbf{L}j^\ast E} & {\mathbf{L}j^\ast E^\prime} & {(Q\oplus Q[1])[1]} \\
            Q & {\mathbf{L}j^\ast E} & {E^{\prime \prime}} & {Q[1]}
            \arrow[from=1-1, to=1-2]
            \arrow[from=1-1, to=2-1]
            \arrow[from=1-2, to=1-3]
            \arrow[from=1-2, to=2-2]
            \arrow["{=}", from=1-3, to=1-4]
            \arrow[from=1-3, to=2-3]
            \arrow[from=1-4, to=2-4]
            \arrow[from=2-1, to=2-2]
            \arrow[from=2-1, to=3-1]
            \arrow[from=2-2, to=2-3]
            \arrow["{=}", from=2-2, to=3-2]
            \arrow[from=2-3, to=2-4]
            \arrow[from=2-3, to=3-3]
            \arrow[from=2-4, to=3-4]
            \arrow[from=3-1, to=3-2]
            \arrow[from=3-2, to=3-3]
            \arrow[from=3-3, to=3-4]
        \end{tikzcd}
    \end{displaymath}
    whose rows and columns are distinguished triangles. 
    As $Q,\mathbf{L}j^\ast E\in D^{\leq 0}_{\operatorname{qc}}(\mathcal{U})$ and $\mathcal{H}^0(Q)\to \mathcal{H}^0(\mathbf{L}j^\ast E)$ is surjective, we have $\tau^{\geq 0}E^{\prime \prime} \cong 0$. 
    Moreover, $\tau^{\geq 0}(Q[2])\cong 0$ because $Q[2]\in D^{\leq 0}_{\operatorname{qc}}(\mathcal{U})$. 
    Hence, $\tau^{\geq 0}\mathbf{L}j^\ast E^\prime \cong 0$, which implies that $\mathcal{H}^i(E^\prime)$ is supported on $T^\prime$ if $i\geq 0$. 
    
    There exists a long exact sequence of cohomology sheaves associated to \eqref{eq:pseudo_approximation_along_etale_nbhd_extended}. 
    This implies $E^\prime \in D^{\leq 0}_{\operatorname{qc}}(\mathcal{X})$. In particular, we have an exact sequence
    \begin{displaymath}
        \mathcal{H}^0 (P) \xrightarrow{f} \mathcal{H}^0 (E) \xrightarrow{g} \mathcal{H}^0 (E^\prime) \to 0.
    \end{displaymath}
    We want to show that $\mathcal{H}^0 (E^\prime)$ is finitely presented. 
    Note that $g$ is surjective. 
    Consider the short exact sequence
    \begin{displaymath}
        0 \to \ker g \to \mathcal{H}^0 (E) \xrightarrow{g} \mathcal{H}^0 (E^\prime) \to 0.
    \end{displaymath}
    Recall that $\mathcal{H}^0 (E)$ is finitely presented. 
    By \cite[\href{https://stacks.math.columbia.edu/tag/0H98}{Tag 0H98}]{StacksProject}, it suffices to show that $\ker g$ is of finite type. 
    Now, exactness says that $\ker g = \operatorname{im} f$. 
    Moreover, we have a short exact sequence
    \begin{displaymath}
        0 \to \ker f \to \mathcal{H}^0 (P) \to\operatorname{im} f \to 0.
    \end{displaymath}
    Since $P$ is pseudocoherent, $\mathcal{H}^0 (P)$ is finitely presented \cite[\href{https://stacks.math.columbia.edu/tag/08FX}{Tag 08FX}]{StacksProject}. 
    Hence, $\operatorname{im} f$ must be of finite type, and so $\mathcal{H}^0 (E^\prime)$ is finitely presented. 

    Now, we finish the proof. 
    As $f$ is concentrated and flat, \Cref{lem:lisse-etale_pullback_and_qc_pullback} says $\mathbf{L}f^{\ast}E^\prime \in D_{\operatorname{qc}}^{\leq 0}(\mathcal{X}^\prime)$ and $\mathcal{H}^0 (\mathbf{L}f^{\ast}E^\prime)$ is finitely presented. 
    Moreover, $\mathcal{X}^\prime$ satisfies pseudoapproximation by compacts, and so there exists a distinguished triangle
    \begin{displaymath}
        F^\prime \to \mathbf{L}f^{\ast}E^\prime \to H^\prime \to F[1]
    \end{displaymath}
    where $F^\prime \in D_{\operatorname{qc},T^\prime}(\mathcal{X}^\prime)^c \cap D^{\leq 0}_{\operatorname{qc}}(\mathcal{X}^\prime)$ and $\tau^{\geq 0}H^\prime \cong 0$. 
    Recall that \eqref{eq:MV} is a flat Mayer--Vietoris square. 
    Then \cite[Proposition 4.2]{Hall/Rydh:2023} implies there exists an induced morphism $\mathbf{R}f_\ast F^\prime \to E^\prime$ such that the pullback is the composition $\mathbf{L}f^\ast \mathbf{R}f_\ast F^\prime \cong F^\prime \to \mathbf{L}f^\ast E^\prime$. 
    Moreover, by loc.\ cit., we obtain that $F=\mathbf{R}f_{\ast}F^\prime \in D_{\operatorname{qc},T^\prime}(\mathcal{X})^c$. 
    Consider the distinguished triangle
    \begin{displaymath}
        F \to E^\prime \to H \to F[1].
    \end{displaymath}
    It follows that $\mathbf{L}f^\ast H \cong H^\prime$, and so, the flatness of $f$ implies that $0 \cong \tau^{\geq 0}\mathbf{L}f^\ast H \cong \mathbf{L}f^\ast (\tau^{\geq 0}H)$. 
    Since $f$ is faithfully flat, $\tau^{\geq 0} H \cong 0$. 
    Consequently, we obtain a diagram with distinguished rows and columns,
    \begin{displaymath}
        \begin{tikzcd}
            P & {\widetilde{P}} & F & {P[1]} \\
            P & E & {E^\prime} & {P[1]} \\
            0 & H & H & {0.}
            \arrow[from=1-1, to=1-2]
            \arrow["{=}"', from=1-1, to=2-1]
            \arrow[from=1-2, to=1-3]
            \arrow[from=1-2, to=2-2]
            \arrow[from=1-3, to=1-4]
            \arrow[from=1-3, to=2-3]
            \arrow["{=}", from=1-4, to=2-4]
            \arrow[from=2-1, to=2-2]
            \arrow[from=2-1, to=3-1]
            \arrow[from=2-2, to=2-3]
            \arrow[from=2-2, to=3-2]
            \arrow[from=2-3, to=2-4]
            \arrow[from=2-3, to=3-3]
            \arrow[from=2-4, to=3-4]
            \arrow[from=3-1, to=3-2]
            \arrow["{=}"', from=3-2, to=3-3]
            \arrow[from=3-3, to=3-4]
        \end{tikzcd}
    \end{displaymath}
    Thus, $\tilde{P}\in D_{\operatorname{qc},T}(\mathcal{X})^c$ and $\tau^{\geq 0}H \cong 0$, which implies $\tilde{P} \to E$ gives the desired morphism.
\end{proof}

\begin{lemma}
    \label{lem:support_along_finite}
    Let $f\colon \mathcal{Y}\to \mathcal{X}$ be a concentrated morphism of quasi-compact quasi-separated algebraic stacks. 
    If $Z\subseteq |\mathcal{X}|$ is closed with quasi-compact complement, then $\mathbf{R}f_\ast D_{\operatorname{qc},f^{-1}(Z)}(\mathcal{X}) \subseteq D_{\operatorname{qc},Z}(\mathcal{X})$.
\end{lemma}

\begin{proof}
    Denote by $j\colon \mathcal{U}\to \mathcal{X}$ the open immersion associated to $Z$. Consider the fibered square
    \begin{displaymath}
        \begin{tikzcd}
            {\mathcal{Y}\times_{\mathcal{X}} \mathcal{U}} & {\mathcal{U}} \\
            {\mathcal{Y}} & {\mathcal{X}.}
            \arrow["{f^\prime}", from=1-1, to=1-2]
            \arrow["{j^\prime}"', from=1-1, to=2-1]
            \arrow["j", from=1-2, to=2-2]
            \arrow["f"', from=2-1, to=2-2]
        \end{tikzcd}
    \end{displaymath}
    Choose $E\in D_{\operatorname{qc},f^{-1}(Z)}(\mathcal{X})$. Then $\mathbf{L}(j^\prime)^\ast E \cong 0$, and so, $\mathbf{R}f^\prime \mathbf{L}(j^\prime)^\ast E \cong 0$. By flat base change, $\mathbf{L}j^\ast \mathbf{R}f_\ast E \cong 0$. Consequently, $\mathbf{R}f_\ast E \in D_{\operatorname{qc},Z}(\mathcal{X})$, which completes the proof.
\end{proof}

\begin{proposition}
    \label{prop:finite_cover_pseudoapproximation}
    Let $\mathcal{X}$ be a concentrated algebraic stack. 
    Suppose there exists a finite faithfully flat morphism $f\colon U \to \mathcal{X}$ of finite presentation from a quasi-affine scheme.
    If the unit of derived pullback/pushforward adjunction $\mathcal{O}_\mathcal{X} \to \mathbf{R}f_\ast \mathcal{O}_U$ splits, then $\mathcal{X}$ satisfies pseudoapproximation by compacts.
\end{proposition}

\begin{proof}
    Let $Z\subseteq |\mathcal{X}|$ be a closed subset with quasi-compact complement. 
    Choose $E\in D^{\leq 0}_{\operatorname{qc},Z}(\mathcal{X})$ such that $\mathcal{H}^0 (E)$ is finitely presented. 
    If the unit $\mathcal{O}_\mathcal{X} \to \mathbf{R}f_\ast \mathcal{O}_U$ splits, then tensoring with $E$ and applying projection formula yields a splitting $E \to \mathbf{R}f_\ast \mathbf{L}f^\ast E$. 
    By \cite[Lemma 5.15]{DeDeyn/Lank/ManaliRahul:2024b}, the unit $E \to \mathbf{R}f_\ast \mathbf{L}f^\ast E$ of the derived pullback/pushforward adjunction is a split monomorphism. 
    Hence, there exists $\mathbf{R}f_\ast \mathbf{L}f^\ast E \to E$ such that $\mathcal{H}^0 (\mathbf{R}f_\ast \mathbf{L}f^\ast E) \to \mathcal{H}^0 (E)$ is surjective. Since $f$ is finite and flat, $\mathbf{L}f^\ast$ and $\mathbf{R}f_\ast$ are $t$-exact functors with respect to the standard $t$-structures. 
    Moreover, by \Cref{lem:lisse-etale_pullback_and_qc_pullback}, $\mathcal{H}^0 (\mathbf{L}f^\ast E)$ is finitely presented. 
    Also, as $V$ is quasi-affine, it satisfies pseudoapproximation by compacts (see \Cref{ex:quasi-affine_pseudoapprox}). 
    This gives us a morphism $P\to \mathbf{L}f^\ast E$ from some $P\in D_{\operatorname{qc},f^{-1}(Z)}^{\leq 0} (V) \cap D_{\operatorname{qc}}(V)^c$ such that $\mathcal{H}^0 (P)\to \mathcal{H}^0 (\mathbf{L}f^\ast E)$ is surjective. 
    However, $\mathbf{R}f_\ast$ being $t$-exact implies $\mathcal{H}^0 (\mathbf{R}f_\ast P)\to \mathcal{H}^0 (\mathbf{R}f_\ast\mathbf{L}f^\ast E)$ is surjective. 
    Furthermore, $f$ is finite and flat, and so \Cref{lem:support_along_finite} says $\mathbf{R}f_\ast P \in D_{\operatorname{qc},T}(\mathcal{X})^c$. 
    Thus, the morphism $\mathbf{R}f_\ast P \to \mathbf{R}f_\ast \mathbf{L}f^\ast E \to E$ satisfies $\mathcal{H}^0 (\mathbf{R}f_\ast P) \to \mathcal{H}^0 (E)$ being surjective with $\mathbf{R}f_\ast P \in D^{\leq 0}_{\operatorname{qc},T}(\mathcal{X}) \cap D_{\operatorname{qc}}(\mathcal{X})^c$.
\end{proof}

\begin{proof}
    [Proof of \Cref{prop:psuedoapprox}]
    By \Cref{lem:approx_by_comorphismcts_implies_Thomason}, we only have to show that $\mathcal{X}$ satisfies pseudoapproximation by compact complexes.
    
    First, we prove the case $\mathcal{X}$ has quasi-finite and separated diagonal. 
    Define $\mathbb{E}$ to be the strictly full $2$-subcategory of algebraic stacks over $\mathcal{S}$ consisting of algebraic stacks whose structure morphism $\mathcal{X} \to \mathcal{S}$ is representable by algebraic spaces, separated, finitely presented, quasi-finite, and flat. 
    Observe the following facts concerning $\mathbb{E}$:
    \begin{itemize}
        \item The source of every object in $\mathbb{E}$ is quasi-compact quasi-separated as any finitely presented morphism of stacks is quasi-compact quasi-separated by definition.
        \item Every morphism in $\mathbb{E}$ is representable by algebraic spaces (see e.g.\ \cite[Lemma 6.7]{DeDeyn/Lank/ManaliRahul:2025}), and so each morphism in $\mathbb{E}$ is concentrated by \cite[Lemma 2.5(3)]{Hall/Rydh:2017}. 
        In particular, as $\mathcal{S}$ is concentrated, every source of an object in $\mathbb{E}$ is concentrated.
    \end{itemize}
    
    Set $\mathbb{D}$ to be the strictly full $2$-subcategory of $\mathbb{E}$ consisting of objects $(\mathcal{Y}\to \mathcal{X})$ such that for every $(\mathcal{U}\to \mathcal{Y})\in \mathbb{E}$ that is \'{e}tale, one has $\mathcal{U}$ satisfying pseudoapproximation by compacts. 
    We invoke \cite[Theorem E]{Hall/Rydh:2018}\footnote{There is a typo in loc.\ cit.\ known to experts, but we reminder the reader: it suffices to only check (I2) for morphisms that are additionally \textit{flat}.} to show $\mathbb{E}=\mathbb{D}$. 
    To this end, we need to verify the following:
    \begin{enumerate}
        \item \label{item:preferred_cover1} if $(\mathcal{U} \to \mathcal{X})\in\mathbb{E}$ is an open immersion and $\mathcal{X}\in\mathbb{D}$, then $\mathcal{U}\in\mathbb{D}$,
        \item \label{item:preferred_cover2} if $(V \to \mathcal{X})\in \mathbb{E}$ is finite, flat and surjective with affine source, then $\mathcal{X}\in\mathbb{D}$, and
        \item \label{item:preferred_cover3} if $(\mathcal{U} \xrightarrow{i} \mathcal{X})$, $(\mathcal{Y} \xrightarrow{f} \mathcal{X})\in\mathbb{E}$, where $i$ is an open immersion and $f$ is \'{e}tale which form an \'{e}tale neighborhood, then $\mathcal{X}\in \mathbb{D}$ whenever $\mathcal{U}$, $\mathcal{Y}\in\mathbb{D}$.
    \end{enumerate}
    We address these below:
    \begin{itemize}
        \item By construction of $\mathbb{D}$, \eqref{item:preferred_cover1} is trivial. 
        \item Let $(f\colon V \to \mathcal{X})\in \mathbb{E}$ is finite, flat and surjective with affine source. 
        Choose $(g\colon \mathcal{U}\to \mathcal{Y})\in \mathbb{E}$ that is \'{e}tale. 
        Denote by $g^\prime \colon V\times_{\mathcal{X}}\mathcal{Y}\to V$ and $f^\prime \colon V\times_{\mathcal{X}}\mathcal{Y} \to \mathcal{Y}$ the natural projections. 
        By \cite[Proposition 3.1]{Olsson/Starr:2003}, $V\times_{\mathcal{X}}\mathcal{Y}$ is quasi-affine because $g$ is representable by algebraic spaces, separated, finitely presented, and \'{e}tale. 
        Moreover, $f^\prime$ is finite faithfully flat and of finite presentation with quasi-affine scheme source. 
        By \cite[Lemma 6.2]{DeDeyn/Lank/ManaliRahul/Peng:2025}, the unit $\mathcal{O}_{\mathcal{X}} \to \mathbf{R}f_\ast \mathcal{O}_V$ of the derived pullback/pushforward adjunction is a split monomorphism. 
        Then flat base change yields a splitting $\mathcal{O}_{V\times_{\mathcal{X}}\mathcal{Y}} \to \mathbf{R}f^\prime_\ast \mathcal{O}_{\mathcal{Y}}$. 
        By \cite[Lemma 5.15]{DeDeyn/Lank/ManaliRahul:2024b}, the unit $\mathcal{O}_{V\times_{\mathcal{X}}\mathcal{Y}} \to \mathbf{R}f^\prime_\ast \mathcal{O}_{\mathcal{Y}}$ of the derived pullback/pushforward adjunction is a split monomorphism. 
        Thus, by \Cref{prop:finite_cover_pseudoapproximation}, \eqref{item:preferred_cover2} holds.
        \item The base change of a flat Mayer--Vietoris square remains such after arbitrary base change on $\mathcal{X}$. 
        See \cite[Lemma 3.1]{Hall/Rydh:2023}. 
        Hence, \Cref{prop:pseudo_approximation_along_etale_nbhd} ensures that $\mathcal{X}$ satisfies pseudoapproximation by compacts.
    \end{itemize}
    This completes the proof of the first case.

    Lastly, the second case. By \cite[Theorem C]{Hall/Rydh:2015}, any Deligne--Mumford $\mathbb{Q}$-stack is concentrated. 
    Hence, all morphisms between Deligne--Mumford $\mathbb{Q}$-stacks is concentrated. 
    Thus, the same proof above goes through with \cite[Theorem E]{Hall/Rydh:2018} in the case of Deligne--Mumford $\mathbb{Q}$-stacks.
\end{proof}

\section{Classification}
\label{sec:classification}

\begin{definition}
    \label{def:graded_filtration}
    Let $\mathcal{X}$ be a quasi-compact quasi-separated algebraic stack. 
    We say a Thomason filtration $\phi$ on $\mathcal{X}$ is \textbf{supported by perfects} if there is a $\mathcal{S}\subseteq \operatorname{Perf}(\mathcal{X})$ such that for all $n\in \mathbb{Z}$,
    \begin{displaymath}
        \phi(n)=\bigcup_{\substack{j\geq n\\ S\in \mathcal{S}}} \operatorname{supp}(\mathcal{H}^j (S)).
    \end{displaymath}
\end{definition}

\begin{convention}
    For convenience, set $\operatorname{Perf}^{\leq 0}(\mathcal{X}) := \operatorname{Perf}(\mathcal{X}) \cap D^{\leq 0}_{\operatorname{qc}}(\mathcal{X})$ and we let $\operatorname{Perf}^{\leq 0}$ play the role of the fixed preaisle $\mathcal{P}^{\leq 0}$ of \Cref{sec:prelim_t-structures_tensor}. In particular, this means that an aisle $\mathcal{T}^{\leq 0}$ in $D_{\operatorname{qc}}(\mathcal{X})$ is tensor if it is closed under the tensor action of $\operatorname{Perf}^{\leq 0}(\mathcal{X})$.
\end{convention}

\begin{lemma}
    \label{lem:affine_supported_by_perfects}
    Let $\mathcal{X}$ be a concentrated algebraic stack such that pseudoapproximation by compact complexes holds for $\mathcal{X}$. Then:
    \begin{enumerate}
        \item For any closed subset $Z$ of $|\mathcal{X}|$ with quasi-compact complement there is $S_Z \in \operatorname{Perf}^{\leq 0}(\mathcal{X})$ such that $\operatorname{supp}(S_Z) \subseteq Z$ and $\operatorname{supp}(\mathcal{H}^0(S_Z)) = Z$.
        \item Every Thomason filtration on $\mathcal{X}$ is supported by perfects.
    \end{enumerate}
     
\end{lemma}

\begin{proof}
    We prove the first claim. By \cite[Proposition 8.2]{Rydh:2016}, there is a closed immersion $i\colon \mathcal{Z} \to \mathcal{X}$ which identifies $|\mathcal{Z}|$ with $Z$. It follows that $E := i^\ast \mathcal{O}_\mathcal{Z}$ is a quasi-coherent sheaf over $\mathcal{X}$ of finite presentation whose support is precisely $Z$. Applying the pseudoapproximation by compact complexes to the pair $(Z,E)$ we obtain a compact object $P \in \operatorname{Perf}^{\leq 0}(\mathcal{X})$ supported inside $Z$ together with a morphism $P \to E$ inducing an epimorphism $\mathcal{H}^0(P) \to \mathcal{H}^0(E)$. It follows that 
    \begin{displaymath}
        Z = \operatorname{supp}(\mathcal{H}^0(E)) \subseteq \operatorname{supp}(\mathcal{H}^0(P)) \subseteq \operatorname{supp}(E) = Z.
    \end{displaymath}
    Thus, letting $S_Z := P$ verifies the claim.
        
    Next, we prove the second claim. For each $n \in \mathbb{Z}$, let us write the Thomason set $\phi(n)$ as a union $\cup_{i \in I_n}Z_i^n$ of closed subsets with quasi-compact complements. Using the first claim, let 
    \begin{displaymath}
        \mathcal{S} := \{S_{Z_i^n}[n] \mid n \in \mathbb{Z}, i \in I_n\}.
    \end{displaymath}
    It can be verified that $\mathcal{S}$ witnesses that $\phi$ is supported by perfects.
\end{proof}

\begin{lemma}
    \label{lem:cohomology_vanish_from_aisle}
    Let $\mathcal{X}$ be a quasi-compact quasi-separated algebraic stack. 
    Suppose $E\in D_{\operatorname{qc}}(\mathcal{X})$ and $\mathcal{S}\subseteq  D_{\operatorname{qc}}(\mathcal{X})$. 
    If $E$ is an $n$-fold extension of objects from $\mathcal{S}$, then for any $j\in \mathbb{Z}$ one has 
    \begin{displaymath}
       \operatorname{supp}(\mathcal{H}^j (E))\subseteq \bigcup_{S\in  \mathcal{S}} \operatorname{supp}(\mathcal{H}^j (S)).
    \end{displaymath}
\end{lemma}

\begin{proof}
    To start, we make a reduction. 
    Let $U \to \mathcal{X}$ be a smooth surjection from a scheme. 
    Recall that the derived pullback is $t$-exact with respect to the standard $t$-structure and takes distinguished triangles to distinguished triangles. 
    So, we can replace $\mathcal{X}$ by a scheme; that is, reducing to the case of a scheme $X$. 

    We prove the desired claim by induction on $n$. 
    The base case with $n=1$ is straightforward. 
    Indeed, $E$ would be isomorphic to an object in $\mathcal{S}$. 
    Now suppose we have proven the claim for any object that is a $k$-fold extension of objects from $\mathcal{S}$ where $1\leq k \leq  n-1$. 
    
    Let $E$ be an $n$-fold extension of objects from $\mathcal{S}$. 
    This means there is a distinguished triangle
    \begin{displaymath}
        E^\prime \to E \to E^{\prime \prime} \to E^{\prime}[1]
    \end{displaymath}
    where $E^\prime$ is an $(n-1)$-fold extension of objects from $\mathcal{S}$ and $E^{\prime \prime}$ is isomorphic to an object from $\mathcal{S}$. 

    Consider the long exact sequence in cohomology 
    \begin{displaymath}
        \cdots \to \mathcal{H}^j (E^\prime) \to \mathcal{H}^j (E) \to \mathcal{H}^j (E^{\prime \prime}) \to \cdots.
    \end{displaymath}
    Let $p \in X\setminus (\cup_{S\in  \mathcal{S}} \operatorname{supp}(\mathcal{H}^j (S)))$. 
    Passing to the stalk of $X$ at $p$, we see from the long exact sequence above that $(\mathcal{H}^j (E))_p =0$. 
    Hence, $p\not \in \operatorname{supp}(\mathcal{H}^j (E))$ as desired. 
    To see, use the inductive hypothesis that $(\mathcal{H}^j (E^\prime))_p=0$. 
    This completes the proof.
\end{proof}

\begin{lemma}
    \label{lem:thomason_description}
    Let $X=\operatorname{Spec}(R)$ for a commutative ring (which could be not Noetherian). Suppose $\mathcal{P}\subseteq \operatorname{Perf}(X)$. 
    Under the bijection of \cite[Theorem 5.1]{Hrbek:2020}, the corresponding Thomason filtration of $\overline{\langle \mathcal{P}\rangle}^{(-\infty,0]}$ can be also identified as Thomason filtration on $X$ given by the rule
    \begin{displaymath}
        n \mapsto \bigcup_{\substack{j\geq n \\ P \in \mathcal{P}}} \operatorname{supp}(\mathcal{H}^j (P)).
    \end{displaymath}
\end{lemma}

\begin{proof}
    Set $\mathcal{U}:= \overline{\langle \mathcal{P}\rangle}^{(-\infty,0]}$. 
    By \cite[Theorem 5.1]{Hrbek:2020}, the corresponding Thomason filtration of $\mathcal{U}$ is given by the rule
    \begin{displaymath}
        k\mapsto \bigcup_{Z_{\mathcal{I}} \textrm{ s.t. } (i_{\mathcal{I}})_\ast \mathcal{O}_{Z_{\mathcal{I}}} [ -k]\in \mathcal{U}} Z_{\mathcal{I}} 
    \end{displaymath}
    where $i_{\mathcal{I}} \colon Z_{\mathcal{I}} \to X$ is a closed subscheme associated to a coherent ideal sheaf $\mathcal{I}\subseteq \mathcal{O}_X$. 
    Denote by $\mathcal{D}$ for the collection of pairs $(\mathcal{I},-k)$ with $\mathcal{I}$ an ideal sheaf and $k$ any integer appearing above. Fix $n\in \mathbb{Z}$. 
    It suffices to check that 
    \begin{displaymath}
        \bigcup_{\substack{j\geq n \\ P \in \mathcal{P}}} \operatorname{supp}(\mathcal{H}^j (P)) = \bigcup_{Z_{\mathcal{I}} \textrm{ s.t. } (i_{\mathcal{I}})_\ast \mathcal{O}_{Z_{\mathcal{I}}} [ -n]\in \mathcal{U}} Z_{\mathcal{I}}.
    \end{displaymath}

    First, let $Z_{\mathcal{I}}$ be a closed subset such that $(i_{\mathcal{I}})_\ast \mathcal{O}_{Z_{\mathcal{I}}} [ -n]\in \mathcal{U}$ for some $(\mathcal{I},-n)\in \mathcal{D}$. 
    Now \cite[Lemma 5.1]{Hrbek:2020} tells us $\mathcal{O}_{Z_{\mathcal{I}}} [ -n]$ is a direct summand of an object $A_{\mathcal{I},n}$ which is an $r$-fold extension of objects from the collection $\{ P [t] : t\geq 0, P \in \mathcal{P} \}$. 
    Note that $\mathcal{O}_{Z_{\mathcal{I}}} [ -n]$ is concentrated in degree $n$. 
    Then \Cref{lem:cohomology_vanish_from_aisle} ensures that 
    \begin{displaymath}
        \begin{aligned}
            Z_{\mathcal{I}} 
            &= \operatorname{supp} (\mathcal{H}^n (\mathcal{O}_{Z_{\mathcal{I}}} [ -n])) 
            \\&\subseteq \operatorname{supp}(\mathcal{H}^n (A_{\mathcal{I},n})) \\&\subseteq \bigcup_{\substack{t\geq 0 \\ P \in \mathcal{P}}} \operatorname{supp}(\mathcal{H}^n (P[t])) 
            \\&\subseteq \bigcup_{\substack{j\geq n \\ P \in \mathcal{P}}} \operatorname{supp}(\mathcal{H}^j (P)).
        \end{aligned}
    \end{displaymath}
    This shows that 
    \begin{displaymath}
        \bigcup_{Z_{\mathcal{I}} \textrm{ s.t. } (i_{\mathcal{I}})_\ast \mathcal{O}_{Z_{\mathcal{I}}} [ -n]\in \mathcal{U}} Z_{\mathcal{I}} \subseteq \bigcup_{\substack{j\geq n \\ P \in \mathcal{P}}} \operatorname{supp}(\mathcal{H}^j (P)).
    \end{displaymath}

    Lastly we check the reverse inclusion. 
    Choose some $P\in \mathcal{P}$ and $j\geq n$ such that $\operatorname{supp}(\mathcal{H}^j (P))$ is nonempty. 
    By \cite[Theorem 5.1]{Hrbek:2020}, we know that 
    \begin{displaymath}
        \overline{ \langle \{ (i_{\mathcal{I}})_\ast \mathcal{O}_{\mathcal{Z}_{\mathcal{I}}} [-k] : (\mathcal{I},k)\in \mathcal{D}\} \rangle}^{(-\infty,0]} = \mathcal{U}.
    \end{displaymath}
    Again, from \cite[Lemma 5.1]{Hrbek:2020}, $P$ is a direct summand of an object $B$ which is an $r$-fold extension of objects from the collection
    \begin{displaymath}
        \{ (i_{\mathcal{I}})_\ast \mathcal{O}_{\mathcal{Z}_{\mathcal{I}}} [-k+s] : (\mathcal{I},-k)\in \mathcal{D}, s\geq 0\}.
    \end{displaymath} 
    Applying \Cref{lem:cohomology_vanish_from_aisle}, we see that
    \begin{equation}
        \label{eq:thomason_description}
        \begin{aligned}
            \emptyset\not=\operatorname{supp}(\mathcal{H}^j (P))
            &\subseteq \operatorname{supp}(\mathcal{H}^j (B))
            \\&\subseteq \bigcup_{\substack{(\mathcal{I},-k)\in \mathcal{D} \\ s\geq 0}} \operatorname{supp}(\mathcal{H}^j ((i_{\mathcal{I}})_\ast \mathcal{O}_{\mathcal{Z}_{\mathcal{I}}} [-k+s])).
        \end{aligned}
    \end{equation}
    Choose any $(\mathcal{I},-k)\in \mathcal{D}$ and $s\geq 0$ above such that 
    \begin{displaymath}
        \operatorname{supp}(\mathcal{H}^j (P)) \cap \operatorname{supp}(\mathcal{H}^j ((i_{\mathcal{I}})_\ast \mathcal{O}_{\mathcal{Z}_{\mathcal{I}}} [-k+s]))\not=\emptyset.
    \end{displaymath}
    Note there is at least one such pair from the hypothesis. This means
    \begin{displaymath}
        \mathcal{H}^j ((i_{\mathcal{I}})_\ast \mathcal{O}_{\mathcal{Z}_{\mathcal{I}}} [-k+s])\not=0,
    \end{displaymath}
    which can only be the case when $j = k -s$. However, $j\geq n$, and so, $-k + s \leq -n$.
    From aisles being closed under positive shifts and $(i_{\mathcal{I}})_\ast \mathcal{O}_{\mathcal{Z}_{\mathcal{I}}} [-k] \in \mathcal{U}$, we have that $(i_{\mathcal{I}})_\ast \mathcal{O}_{\mathcal{Z}_{\mathcal{I}}} [-n]\in \mathcal{U}$. 
    It is clear that the support of $\mathcal{H}^j ((i_{\mathcal{I}})_\ast \mathcal{O}_{\mathcal{Z}_{\mathcal{I}}} [-k+s])$ is $Z_{\mathcal{I}}$. 
    So, $\operatorname{supp}(\mathcal{H}^j (P))$ is contained in the union of all such $Z_{\mathcal{I}}$ above in light of \eqref{eq:thomason_description}. 
    In particular, we see that
    \begin{displaymath}
        \operatorname{supp}(\mathcal{H}^j (P)) \subseteq \bigcup_{Z_{\mathcal{I}} \textrm{ s.t. } (i_{\mathcal{I}})_\ast \mathcal{O}_{Z_{\mathcal{I}}} [ -n]\in \mathcal{U}} Z_{\mathcal{I}},
    \end{displaymath}
    which completes the proof.
\end{proof}

\begin{remark}
    \label{rmk:proxy_aisle}
    Let $\mathcal{X}$ be a concentrated algebraic stack with quasi-finite and separated diagonal (resp.\ a Deligne--Mumford $\mathbb{Q}$-stack). 
    Consider any set $\mathcal{C}\subseteq D_{\operatorname{qc}}(\mathcal{X})$. 
    By \Cref{prop:psuedoapprox}, $D^{\leq 0}_{\operatorname{qc}}(\mathcal{X})$ is compactly generated by $D^{\leq 0}_{\operatorname{qc}}(\mathcal{X}) \cap \operatorname{Perf}(\mathcal{X})$. 
    Then, from \cite[Lemmas 2.6 \& 2.8]{Hrbek/Lank/LeGros/Pavon:2026}, it follows
    \begin{displaymath}
        \overline{\langle \mathcal{C} \rangle}^{(-\infty,0]}_{\otimes} = \overline{\langle (D^{\leq 0}_{\operatorname{qc}}(\mathcal{X}) \cap \operatorname{Perf}(\mathcal{X})) \otimes^{\mathbf{L}} \mathcal{C} \rangle}^{(-\infty,0]}.
    \end{displaymath}
\end{remark}

\begin{lemma}
    \label{lem:pullback_t_exactness}
    Let $\mathcal{X}$ be a concentrated algebraic stack with quasi-finite and separated diagonal (resp.\ a Deligne--Mumford $\mathbb{Q}$-stack). 
    Suppose $f\colon U \to \mathcal{X}$ is a flat morphism from an affine scheme. 
    Then, for any $\mathcal{P}\subseteq \operatorname{Perf}(X)$ such that $\operatorname{preaisle}_{\otimes}(\mathcal{P})$ is an aisle on $D_{\operatorname{qc}}(\mathcal{X})$, one has that $\mathbf{L} f^\ast$ induces a $t$-exact functor
    \begin{displaymath}
        (D_{\operatorname{qc}}(\mathcal{X}),\overline{\langle \mathcal{P} \rangle}^{(-\infty,0]}_{\otimes}) \to (D_{\operatorname{qc}}(U),\overline{ \langle \mathbf{L} f^\ast \mathcal{P} \rangle }^{(-\infty,0]}_{\otimes}).
    \end{displaymath}
\end{lemma}

\begin{proof}
    We prove the desired claim by contradiction. 
    Since $U$ is affine, all aisles are tensor compatible. 
    By \Cref{rmk:proxy_aisle}, it is straightforward to see that $\mathbf{L}f^\ast \overline{\langle \mathcal{P} \rangle}^{(-\infty,0]}_{\otimes} \subseteq \overline{ \langle \mathbf{L} f^\ast \mathcal{P} \rangle}^{(-\infty,0]}$, which tells us $\mathbf{L}f^\ast$ is right $t$-exact. 
    Indeed, from \Cref{lem:dubey_sahoo_corrected}, it suffices to check that 
    \begin{displaymath}
        \mathbf{L}f^\ast ((D^{\leq 0}_{\operatorname{qc}}(\mathcal{X}) \cap \operatorname{Perf}(\mathcal{X})) \otimes^{\mathbf{L}} \mathcal{P})
        \subseteq \overline{ \langle \mathbf{L} f^\ast \mathcal{P} \rangle}^{(-\infty,0]} 
    \end{displaymath}
    Since $\mathbf{L}f^\ast (D^{\leq 0}_{\operatorname{qc}}(\mathcal{X}) \cap \operatorname{Perf}(\mathcal{X})) \subseteq \overline{\langle \mathcal{O}_U \rangle}^{(-\infty,0]}$, the claim follows. 
    Hence, the assumption ensures $\mathbf{L}f^\ast$ must fail to be left $t$-exact. 
    So, we can find an $E\in \operatorname{coaisle}_{\otimes}(\mathcal{P})$ such that $\mathbf{L}f^\ast E \not\in \operatorname{coaisle}_{\otimes}(\mathbf{L} f^\ast \mathcal{P})$. 
    This tells us $\mathbf{L} f^\ast E [-1]\not\in (\overline{\langle \mathbf{L} f^\ast \mathcal{P}\rangle}^{(-\infty,0]}_{\otimes})^\perp$. 
    By \cite[Lemma 2.2]{DeDeyn/Lank/ManaliRahul/Peng:2025}, we know that $\mathbf{L}f^\ast \mathcal{P}\subseteq \operatorname{Perf}(U)$. 
    Moreover, using \cite[Lemma 3.1]{AlonsoTarrio/Lopez/Salorio:2003}, we can find an $A\in \mathcal{P}$ and $j\geq 0$ such that $\operatorname{Hom}(\mathbf{L}f^\ast A[j] , \mathbf{L}f^\ast E[-1])\not=0$. 
    Now we use adjunction to see that,
    \begin{displaymath}
        \begin{aligned}
            \operatorname{Hom} & (\mathbf{L}f^\ast A[j] , \mathbf{L}f^\ast E[-1]) 
            \\&\cong \operatorname{Hom}(\mathcal{O}_U, \operatorname{\mathbf{R}\mathcal{H}\! \mathit{om}}(\mathbf{L}f^\ast A[j], \mathbf{L}f^\ast E[-1]) ) && \textrm{\cite[\href{https://stacks.math.columbia.edu/tag/08DH}{Tag 08DH}]{StacksProject}}
            \\&\cong \operatorname{Hom}(\mathcal{O}_U, \mathbf{L}f^\ast \operatorname{\mathbf{R}\mathcal{H}\! \mathit{om}}(A[j], E[-1]) ) && \textrm{\cite[Lemma 2.3]{DeDeyn/Lank/ManaliRahul/Peng:2025}}.
        \end{aligned}
    \end{displaymath}
    This implies that
    \begin{displaymath}
        \mathbf{L}f^\ast \operatorname{\mathbf{R}\mathcal{H}\! \mathit{om}}(A[j], E[-1]) \not\in D^{\geq 1}_{\operatorname{qc}}(U),
    \end{displaymath}
    and because $\mathbf{L}f^\ast$ is $t$-exact with respect to the standard $t$-structures, 
    \begin{displaymath}
        \operatorname{\mathbf{R}\mathcal{H}\! \mathit{om}}(A[j], E[-1]) \not\in D^{\geq 1}_{\operatorname{qc}}(\mathcal{X}).
    \end{displaymath}
    Then we can find a $B\in D^{\leq 0}_{\operatorname{qc}}(\mathcal{X})$ such that
    \begin{displaymath}
        \begin{aligned}
            0 &\not= 
            \operatorname{Hom}(B , \operatorname{\mathbf{R}\mathcal{H}\! \mathit{om}}(A[j], E[-1]) )
            \\&\cong \operatorname{Hom}(B \otimes^{\mathbf{L}} A[j], E[-1] ) && (\textrm{\cite[\href{https://stacks.math.columbia.edu/tag/08DH}{Tag 08DH}]{StacksProject}}).
        \end{aligned}
    \end{displaymath}
    However, $E[-1]$ is in $(\overline{\langle \mathcal{P}\rangle}^{(-\infty,0]}_{\otimes})^\perp$ and $B\otimes^{\mathbf{L}} A[j]\in \overline{\langle \mathcal{P}\rangle}^{(-\infty,0]}_{\otimes}$, which is a contradiction.
\end{proof}

\begin{proof}
    [Proof of \Cref{thm:stacky_DB}]
    Before getting to work, we make a few observations that will be used freely in the proof. 
    First, our constraints on $\mathcal{X}$ ensures that $D_{\operatorname{qc}}(X)^c$ (i.e.\ the compacts) coincide with $\operatorname{Perf}(\mathcal{X})$. 
    Secondly, as $\mathcal{X}$ is concentrated and $1$-Thomason (see \cite[Theorem A]{Hall/Rydh:2017}), we have that $D_{\operatorname{qc}}(\mathcal{X})$ is a rigidly compactly generated tensor triangulated category. 
    Thirdly, from \Cref{lem:dubey_sahoo_corrected} (which uses that $D^{\leq 0}_{\operatorname{qc}}(\mathcal{X})$ is compactly generated), it follows that $\operatorname{preaisle}_{\otimes} (\mathcal{P})$ is a $\otimes$-aisle on $D_{\operatorname{qc}}(\mathcal{X})$ for all $\mathcal{P}\subseteq \operatorname{Perf}(\mathcal{X})$. 
    That is, $\overline{\langle \mathcal{P} \rangle}^{(-\infty,0]}_{\otimes}$ exists for all $\mathcal{P}\subseteq \operatorname{Perf}(\mathcal{X})$. Finally, since $\mathcal{X}$ satisfies the pseudoapproximation by compact complexes by the (proof of) \Cref{prop:psuedoapprox}, \Cref{lem:affine_supported_by_perfects} yields that any Thomason filtration on $\mathcal{X}$ is supported by perfects.

    Now we move onwards to our proof. 
    This will be done by constructing a series of maps between the desired collections. 
    Let $f\colon U \to \mathcal{X}$ be a smooth surjective morphism from an affine scheme (which exists by \textrm{\cite[\href{https://stacks.math.columbia.edu/tag/04YC}{Tag 04YC}]{StacksProject}}). 
    To start, we show there is an injective mapping
    \begin{equation}
        \label{eq:classification_map1}
        \begin{aligned}
            \{ \otimes\textrm{-aisle on } & D_{\operatorname{qc}}(\mathcal{X}) \textrm{ generated by } \mathcal{P} \subseteq \operatorname{Perf}(\mathcal{X}) \} 
            \\&\to \{ \otimes\textrm{-aisle on } D_{\operatorname{qc}}(U) \textrm{ generated by } \mathcal{Q} \subseteq \operatorname{Perf}(U) \}.
        \end{aligned}
    \end{equation}
    This part of the proof is done by contradiction. 
    Suppose $\mathcal{P},\mathcal{P}^\prime \subseteq \operatorname{Perf}(\mathcal{X})$ satisfy $\overline{\langle \mathcal{P} \rangle}^{(-\infty,0]}_{\otimes} \not= \overline{\langle \mathcal{P}^\prime \rangle}^{(-\infty,0]}_{\otimes}$ and $\overline{\langle \mathbf{L} f^\ast \mathcal{P} \rangle}^{(-\infty,0]}_{\otimes} = \overline{\langle \mathbf{L} f^\ast \mathcal{P}^\prime \rangle}^{(-\infty,0]}_{\otimes}$. 
    Then we can find an $E\in \overline{\langle \mathcal{P} \rangle}^{(-\infty,0]}_{\otimes}$ such that $E\not\in \overline{\langle \mathcal{P}^\prime \rangle}^{(-\infty,0]}_{\otimes}$. 
    Consider the truncation triangle with respect to the aisle $\overline{\langle \mathcal{P}^\prime \rangle}^{(-\infty,0]}_{\otimes}$,
    \begin{displaymath}
        \tau^{\leq 0}_{\mathcal{P}^\prime} E \to E \to \tau^{\geq 1}_{\mathcal{P}^\prime} \to (\tau^{\leq 0}_{\mathcal{P}^\prime} E).
    \end{displaymath}
    By \Cref{lem:pullback_t_exactness}, the following
    \begin{displaymath}
        \mathbf{L} f^\ast \tau^{\leq 0}_{\mathcal{P}^\prime} E \to \mathbf{L} f^\ast E \to \mathbf{L} f^\ast \tau^{\geq 1}_{\mathcal{P}^\prime} \to \mathbf{L} f^\ast (\tau^{\leq 0}_{\mathcal{P}^\prime} E)
    \end{displaymath}
    is the truncation triangle for $\mathbf{L} f^\ast E$ with respect to $\overline{\langle \mathbf{L} f^\ast \mathcal{P}^\prime \rangle}^{(-\infty,0]}_{\otimes}$. 
    However, the assumption tells us that $\mathbf{L} f^\ast E \in \overline{\langle \mathbf{L} f^\ast \mathcal{P}^\prime \rangle}^{(-\infty,0]}_{\otimes}$, and so, $\mathbf{L} f^\ast \tau^{\geq 1}_{\mathcal{P}^\prime} E= 0$. 
    From flatness of $f$, it follows that $\tau^{\geq 1}_{\mathcal{P}^\prime} E = 0$. 
    This tells us that $E \in \overline{\langle \mathcal{P}^\prime \rangle}^{(-\infty,0]}_{\otimes}$, which is absurd given the assumption. 
    Hence, we have the desired injection for \eqref{eq:classification_map1}.

    Next we show there is an injective mapping
    \begin{equation}
        \label{eq:classification_map2}
        \begin{aligned}
            \{ \otimes\textrm{-aisle on } & D_{\operatorname{qc}}(\mathcal{X}) \textrm{ generated by } \mathcal{P} \subseteq \operatorname{Perf}(\mathcal{X}) \} 
            \\&\to \{ \textrm{Thomason filtrations on } U \}
        \end{aligned}
    \end{equation}
    which factors through \eqref{eq:classification_map1}. 
    From \cite[Theorem 5.1]{Hrbek:2020}, we have an injective mapping the image of \eqref{eq:classification_map1} into collection of Thomason filtrations on $U$. Moreover, \Cref{lem:affine_supported_by_perfects} tells us each Thomason filtration on $U$ is supported by perfects. 
    Specifically, for each $\mathcal{P}\subseteq \operatorname{Perf}(\mathcal{X})$, we assign $\overline{\langle \mathbf{L} f^\ast \mathcal{P} \rangle}^{(-\infty,0]}_{\otimes}$ to the Thomason filtration on $U$ given by 
    \begin{displaymath}
        \phi_{\mathbf{L} f^\ast \mathcal{P}} \colon n\mapsto \bigcup_{\substack{j\geq n\\ P\in \mathcal{P}}} \operatorname{supp}(\mathcal{H}^j ( \mathbf{L}f^\ast P)).
    \end{displaymath}
    The presentation above follows from \Cref{lem:thomason_description}. 
    So, tying things to together, we have the desired mapping in \eqref{eq:classification_map2}.

    Now we claim there is a bijective mapping (i.e.\ the desired one-to-one correspondence)
    \begin{equation}
        \label{eq:classification_map3}
        \begin{aligned}
            \{ \otimes\textrm{-aisle on } & D_{\operatorname{qc}}(\mathcal{X}) \textrm{ generated by } \mathcal{P} \subseteq \operatorname{Perf}(\mathcal{X}) \} 
            \\&\to \{  \textrm{Thomason filtrations on } \mathcal{X}\}
        \end{aligned}
    \end{equation}
    which factors through \eqref{eq:classification_map2}. 
    It suffices to show there is a bijective mapping
    \begin{equation}
        \label{eq:classification_map4}
        \begin{aligned}
            \{  \textrm{Thomason filtrations on } \mathcal{X}\} 
            \\& \to \{ \phi_{\mathbf{L} f^\ast \mathcal{P}} \colon \mathcal{P}\subseteq \operatorname{Perf}(\mathcal{X}) \}.
        \end{aligned}
    \end{equation}
    Indeed, the image of the injective mapping in \eqref{eq:classification_map2} is precisely the target in \eqref{eq:classification_map4}. 
    To see, by construction, the image of \eqref{eq:classification_map2} lands into $\{ \phi_{\mathbf{L} f^\ast \mathcal{P}} \colon \mathcal{P}\subseteq \operatorname{Perf}(\mathcal{X}) \}$, so we have an injective map
    \begin{displaymath}
        \begin{aligned}
            \{ \otimes\textrm{-aisle on } D_{\operatorname{qc}}(\mathcal{X}) \textrm{ generated by } & \mathcal{P} \subseteq \operatorname{Perf}(\mathcal{X}) \} 
            \\&\to \{ \phi_{\mathbf{L} f^\ast \mathcal{P}} \colon \mathcal{P}\subseteq \operatorname{Perf}(\mathcal{X}) \}.
        \end{aligned}
    \end{displaymath}
    On the other hand, for each $\mathcal{P}\subseteq \operatorname{Perf}(\mathcal{X})$, we know $\overline{\langle \mathcal{P} \rangle}^{(-\infty,0]}_{\otimes}$ exists on $D_{\operatorname{qc}}(\mathcal{X})$, so its image under \eqref{eq:classification_map2} is exactly $\phi_{\mathbf{L} f^\ast \mathcal{P}}$, showing the mapping is also surjective.
    
    Let $\psi$ be a Thomason filtration on $\mathcal{X}$. Since it is supported by perfects, we can find $\mathcal{Q}\subseteq \operatorname{Perf}(\mathcal{X})$ such that 
    \begin{displaymath}
        \psi(n)=\bigcup_{\substack{j\geq n\\ Q\in \mathcal{Q}}} \operatorname{supp}(\mathcal{H}^j (Q)).
    \end{displaymath}
    Using \cite[Lemma 4.8]{Hall/Rydh:2017}, we can see that
    \begin{displaymath}
        \begin{aligned}
            \bigcup_{\substack{j\geq n\\ Q\in \mathcal{Q}}} \operatorname{supp}(\mathcal{H}^j (\mathbf{L}f^\ast Q))
            &= \bigcup_{\substack{j\geq n\\ Q\in \mathcal{Q}}} f^{-1}(\operatorname{supp}(\mathcal{H}^j (Q)))
            \\&= f^{-1}( \bigcup_{\substack{j\geq n\\ Q\in \mathcal{Q}}} \operatorname{supp}(\mathcal{H}^j (Q))) 
            \\&= f^{-1}(\phi(n)).
        \end{aligned}
    \end{displaymath}
    Using \cite[Lemma 4.8(3)]{Hall/Rydh:2017}, it is straightforward to check that this defines a Thomason filtration on $U$. 
    This Thomason filtration on $U$ is denoted by $f^{-1}(\phi)$. It is given by the rule $n \mapsto f^{-1}(\phi(n))$. 
    In fact, the very construction shows $f^{-1}(\phi) = \phi_{\mathbf{L}f^\ast \mathcal{Q}}$. 
    This gives us a well-defined mapping in \eqref{eq:classification_map4}. 
    We need to check it is bijective. 
    Fortunately, the construction of the mapping shows that it must be surjective, so we only need to check it is injective. 
    However, this follows from the fact that $f\colon U \to \mathcal{X}$ is surjective, completing the proof.
\end{proof}

\bibliographystyle{alpha}
\bibliography{mainbib}

\end{document}